\documentclass[12pt]{amsart}
\usepackage{etex}
\usepackage{graphicx}
\usepackage{enumerate}
\usepackage{empheq}
\usepackage{subfigure}
\usepackage{empheq}
\usepackage{amssymb}
\usepackage{ccaption}
\usepackage[margin=1in]{geometry}
\usepackage{setspace}
\usepackage{pst-all}
\usepackage{array}
\usepackage{tikz}
\usepackage{xcolor}

\newtheorem{theorem}{Theorem}%[section]
\newtheorem*{theorem*}{Theorem}
\newtheorem{lemma}[theorem]{Lemma}
\newtheorem{proposition}[theorem]{Proposition}
\newtheorem{corollary}[theorem]{Corollary}

\theoremstyle{definition}
\newtheorem{definition}[theorem]{Definition}
\newtheorem{example}[theorem]{Example}

\theoremstyle{remark}

\numberwithin{equation}{section}

\newcommand{\svw}[1]{\textcolor{black}{#1}}
\newcommand{\svwii}[1]{\textcolor{black}{#1}}
\newcommand{\svwM}[1]{\textcolor{black}{#1}}

\newcommand{\key}{\kappa}
\newcommand{\fund}{\mathfrak{F}}

\newcommand{\QKT}{\mathrm{QKT}}
\newcommand{\KT}{\mathrm{KT}}

\newcommand{\wt}{\mathrm{wt}}

\renewcommand{\a}{\mathbf{a}}
\renewcommand{\b}{\mathbf{b}}
\newlength\cellsize 

\savebox2{%
\begin{picture}(16,16)
\put(0,0){\line(1,0){16}}
\put(0,0){\line(0,1){16}}
\put(16,0){\line(0,1){16}}
\put(0,16){\line(1,0){16}}
\end{picture}}

\newcommand\cellify[1]{\def\thearg{#1}\def\nothing{}%
\ifx\thearg\nothing\vrule width0pt height\cellsize depth0pt%
  \else\hbox to 0pt{\usebox2\hss}\fi%
  \vbox to \cellsize{\vss\hbox to \cellsize{\hss$_{#1}$\hss}\vss}}

\newcommand\tableau[1]{\vtop{\let\\=\cr
\setlength\baselineskip{-16000pt}
\setlength\lineskiplimit{16000pt}
\setlength\lineskip{0pt}
\halign{&\cellify{##}\cr#1\crcr}}}

\begin{document}

\title[Slide multiplicity free key polynomials]{Slide multiplicity free key polynomials}

\author{Soojin Cho}
\address{Department of Mathematics, Ajou University, Suwon  16499,  Korea}
\email{chosj@ajou.ac.kr}

\author{Stephanie van Willigenburg}
\address{Department of Mathematics, University of British Columbia, Vancouver BC V6T 1Z2, Canada}
\email{steph@math.ubc.ca}

\thanks{This research was supported by the Ajou University research fund. This research was also supported in part by the National Sciences and Engineering Research Council of Canada.}

\begin{abstract}
Schubert polynomials are refined by the key polynomials of Lascoux-Sch\"{u}tzenberger, which in turn are refined by the fundamental slide polynomials of Assaf-Searles. In this paper we determine which fundamental slide polynomial refinements of key polynomials, indexed by strong compositions, are multiplicity free. We also give a recursive algorithm to determine all terms in the fundamental slide polynomial refinement of a key polynomial indexed by a strong composition. From here, we apply our results to begin to classify which fundamental slide polynomial refinements, indexed by weak compositions, are multiplicity free. We completely resolve the cases when the weak composition has at most two nonzero parts or the sum has at most two nonzero terms.
\end{abstract}

\date{\today}

\maketitle

%%%%%%%%%%%%%%%%%%%%%%%%%%%%%%%%%%%
\section{Introduction} \label{sec:intro}

Schubert polynomials were introduced by Lascoux-Sch\"{u}tzenberger \cite{LS82}, and represent cohomology classes of Schubert cycles in flag varieties. They are also generalizations of the well-known Schur polynomials and are generating functions for pipe dreams. Lascoux-Sch\"{u}tzenberger showed in \cite{LS90} that Schubert polynomials expand as a positive linear combination of key polynomials, which are not only a tool for studying Schubert polynomials but also arise as characters for the Demazure modules of type $A$ \cite{I03, Mas09, RS95}. Recently multiplicity free key polynomials in terms of monomials have arisen in the study of spherical Schubert geometry \cite{HY20} and multiplicity free key polynomials in terms of monomials were recently succinctly classified by Hodges-Yong in \cite{HY20MF}.

This result joins a myriad of other multiplicity free classifications in algebraic combinatorics including multiplicity free products of Schur functions in terms of Schur functions \cite{Stem01}, and analogously for multiplicity free Schur $P$-function products \cite{Bess02}, multiplicity free skew Schur functions in terms of Schur functions \cite{Gut10, TY10}, multiplicity free Schur $P$-functions in terms of Schur functions \cite{SvW07}, multiplicity free Stanley symmetric functions in terms of monomials \cite{BP14}, multiplicity free Schubert polynomials in terms of monomials \cite{FM20}, and multiplicity free Schur functions in terms of fundamental quasisymmetric functions \cite{BvW13}. This latter result is as follows.

\begin{theorem}[Theorem 3.3, \cite{BvW13}]\label{thm:BvW13} 
For $\lambda$  a partition of $n$,  the Schur function $s_\lambda$ has a multiplicity free expansion into a sum of fundamental quasisymmetric functions if and only if $\lambda$ or its transpose is one of 
\begin{enumerate}
    \item[(1)] $(3,3)$ if $n=6$,
    \item[(2)] $(4,4)$ if $n=8$,
    \item[(3)] $(n-2, 2)$ if $n\geq 4$,
    \item[(4)] $(n-k, 1^k)$ if $n\geq 1$ and $0\leq k\leq n-1$.
\end{enumerate}
\end{theorem}

Fundamental quasisymmetric functions themselves are an active area of study, being the weight enumerators of chains in the theory of $P$-partitions \cite{Ges84} and being isomorphic to the irreducible characters of the $0$-Hecke algebra \cite{French96}. They are also the stable limits of fundamental slide polynomials \cite[Theorem 4.4]{AS18} that were introduced by Assaf-Searles to study Schubert polynomials \cite{AS17} who similarly showed that key polynomials are the stable limits of Schur functions \cite[Corollary 4.9]{AS18}.

Since Schur functions and fundamental quasisymmetric functions are stable limits of key polynomials and fundamental slide polynomials respectively,  as described in Theorem~\ref{thm:limit}, it is natural to ask if we can classify key polynomials that are expanded as a multiplicity free sum of fundamental slide polynomials.
In this paper we consider the multiplicity free expansion of key polynomials into fundamental slide polynomials, \svw{and obtain the following classification in Theorem~\ref{thm:main2}.}
\svw{\begin{theorem*} For $\alpha$ a strong composition, the key polynomial $\key _\alpha$ has a multiplicity free expansion into a sum of fundamental slide polynomials if and only if $\alpha=(\alpha_1, \dots, \alpha_\ell)$ satisfies all three of the following conditions: \begin{enumerate}
\item[(a)] There is no $i<j<k$ such that $\alpha_i<\alpha_j<\alpha_k$.
\item[(b)] There is no $i<j<k<l$ such that $\alpha_i, \alpha_j<\alpha_l<\alpha_k$.
\item[(c)] There is no $i<j<k<l$ such that $\alpha_i, \alpha_j+1<\alpha_k=\alpha_l$.
\end{enumerate}
\end{theorem*}}
More precisely our paper is structured as follows. In Section~\ref{sec:definition} we recall the relevant background before classifying when a key \svw{polynomial} is equal to a fundamental slide polynomial in Theorem~\ref{thm:k=f} and the sum of two \svw{fundamental slide polynomials} in Theorem~\ref{thm:2terms}. We then restrict our attention to key polynomials indexed by strong compositions and their multiplicity free expansion into fundamental slide polynomials in \svw{Section~\ref{sec:MFforStrong},} generalizing Theorem~\ref{thm:BvW13}. In particular we classify this in Theorem~\ref{thm:main2} and give an algorithm to  compute the terms in general in Subsection~\ref{subsec:alg}. In Section~\ref{sec:MFforWeak} we work towards a full classification, including when a key polynomial whose index has two nonzero parts is a multiplicity free expansion of fundamental slide polynomials in Subsection~\ref{subsec:two nonzero parts}.

\

\textbf{Acknowledgments} The authors would like to thank Sami Assaf for posing the question and for many helpful conversations. They would also like to thank \svw{the referee for their care and excellent suggestions, and} the Korea Institute for Advanced Study where some of this research took place. 

%%%%%%%%%%%%%%%%%%%%%%%%%%%%%%%%%%%
\section{Tableaux and polynomials} \label{sec:definition}

In this section we recall the background needed in order to state and prove our results.

A \emph{weak composition}  $\a=(a_1, \dots, a_\ell)$ of \emph{length $\ell$} is a finite sequence of nonnegative integers, and a  \emph{strong composition} $\alpha=(\alpha_1, \dots, \alpha_\ell)$ of \emph{length $\ell$} is a finite sequence of positive integers. For a weak composition $\a$, we let $\mathrm{flat}(\a)$ be the strong composition obtained by removing zeros from $\a$ and $\mathrm{sort}(\a)$ be the \svw{\emph{partition}} obtained by rearranging the parts of $\mathrm{flat}(\a)$ into weakly decreasing order. For example, if $\a = (0,0,2,3)$ then $\mathrm{flat}(\a) = (2,3)$ and $\mathrm{sort}(\a) = (3,2)$. Given two weak compositions  $\a=(a_1, \dots, a_\ell)$ and  $\b=(b_1, \dots, b_\ell)$ we say that \emph{$\b$ dominates $\a$} denoted by $\b\geq\a$, if $b_1+\cdots+b_i\geq a_1+\cdots+a_i$ for all $i=1, \dots, \ell$. Given two strong compositions $\alpha=(\alpha_1, \dots, \alpha_\ell)$ and $\beta=(\beta_1, \dots, \beta_{\ell'})$, we say that \emph{$\beta$ refines $\alpha$} if there is a sequence $0=i_0< i_1< i_2<\cdots<i_\ell=\ell'$ such that $\beta_{i_{j-1}+1}+\cdots+\beta_{i_j}=\alpha_j$ for $j=1, \dots, \ell$.

A \emph{diagram} is a finite subset of  $\mathbb N\times \mathbb N$ where the coordinate $(i, j)$ represents the cell at the $i$th row from the bottom and the $j$th column from the left.  

We now come to two central definitions, which define the tableaux that we will be using.

\begin{definition}\label{def:(Q)KT}
For a weak composition $\a=(a_1, \dots, a_\ell)$, a \emph{Kohnert tableau of content $\a$} is a diagram filled with $a_i$ $i$s for $i=1, \dots, \ell$, satisfying the following conditions:
 \begin{enumerate}
 \item[(i)] For each $i=1, \dots, \ell$, there is exactly one $i$ in each column from $1$ through $a_i$.
 \item[(ii)] Each entry in row $i$ is at least $i$  for all $i$.
 \item[(iii)] For each $i=1, \dots, \ell$, the cells filled with $i$ weakly descend from left to right. 
 \item[(iv)] If $i<j$ appear in a column with $i$ above $j$, then there is an $i$ in the column immediately to the right of and strictly above \svw{the cell containing the} $j$.
 \end{enumerate}

A Kohnert tableau of content $\a$ is called  a \emph{quasi-Yamanouchi} Kohnert tableau if it satisfies an additional condition:
 \begin{enumerate}
 \item[(v)] If a row is not empty, say the $i$th row, then
     \begin{enumerate} 
      \item[-] either there is a cell filled with $i$ in the $i$th row,
      \item[-] or there is a cell in the $(i+1)$th row that lies weakly right of a cell in the $i$th row.
    \end{enumerate}
 \end{enumerate}
\end{definition}

\begin{definition}\label{def:QKTset}
For a weak composition $\a=(a_1, \dots, a_\ell)$, let $\KT(\a)$ be the set of Kohnert tableaux  \svw{and}  $\QKT(\a)$ be the set of quasi-Yamanouchi Kohnert tableaux of content $\a$. For $T\in \KT(\a)$ we let $\wt(T)$ be the weak composition whose $i$th part is the number of cells in the $i$th row of $T$.
The \emph{basic} quasi-Yamanouchi Kohnert tableau of content $\a$ is the unique Kohnert tableau $T$ with $a_i$ cells  in row $i$ filled with $i$ for $i=1, \dots, \ell$, namely $\wt(T) = \a$.  
\end{definition}

\begin{example}\label {ex:QKT} Figure~\ref{fig:QKT(0032)} illustrates some Kohnert tableaux that are quasi-Yamanouchi. Note that the leftmost one is basic. Figure~\ref{fig:KT(0032)} illustrates some \svw{Kohnert tableaux that are not quasi-Yamanouchi.}
\end{example}

\begin{figure}[ht]
\begin{center}
    \begin{tikzpicture}[xscale=1.5,yscale=2] 
     \node at (0,0)  {$ \vline\tableau{4&4\\3&3&3 \\ \\ \\ \hline} $};
     \node at (2,0)  {$ \vline\tableau{4\\3&4\\ & 3&3 \\ \\ \hline} $};
     \node at (4,0)  {$ \vline\tableau{4\\3&3&3\\ & 4 \\ \\ \hline} $};
     \node at (6,0)  {$ \vline\tableau{ \\3&3&3\\ 4 & 4 \\  \\ \hline} $};
     \node at (8,0)  {$ \vline\tableau{\\3&3\\ 4& &3 \\ &4&\\ \hline} $};
\end{tikzpicture}
  \end{center}
  \caption{\label{fig:QKT(0032)} Quasi-Yamanouchi Kohnert tableaux of content $(0, 0,3, 2)$.}
\end{figure}
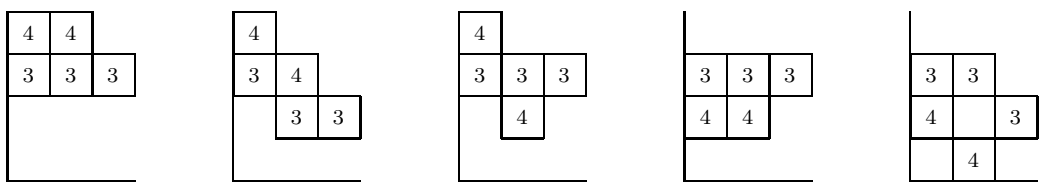

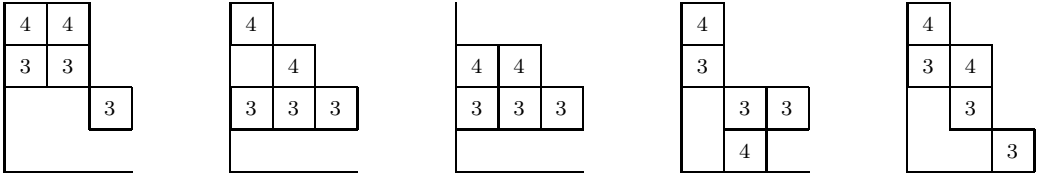
\begin{figure}[ht]
\begin{center}
    \begin{tikzpicture}[xscale=1.5,yscale=2] 
     \node at (0,0)  {$\vline\tableau{4&4\\3&3\\ & & 3 \\ \\ \hline}$};
     \node at (2,0)  {$ \vline\tableau{4\\&4\\ 3& 3&3 \\ \\ \hline} $};
     \node at (4,0)  {$ \vline\tableau{\\4&4\\3&3&3\\ \\ \hline} $};
     \node at (6,0)  {$ \vline\tableau{4\\3\\&3&3\\  & 4 \\   \hline} $};
     \node at (8,0)  {$ \vline\tableau{4\\3&4\\ &3 \\ & &3\\ \hline} $};
\end{tikzpicture}
  \end{center}
  \caption{\label{fig:KT(0032)} Some Kohnert tableaux of content $(0, 0,3, 2)$ that are not quasi-Yamanouchi.}
\end{figure}

For a weak composition $\b=(b_1, \dots, b_\ell)$, we use $x^\b$ to denote  the monomial $x_1^{b_1}  \cdots x_\ell^{b_\ell}$. We are now ready to define key polynomials, fundamental slide polynomials, each of their stable limits, and the relationship between them.

\begin{definition}\label{def:keyfund} For a weak composition  $\a$ of length $\ell$, 
\begin{enumerate}
\item the \emph{key polynomial} $\key_{\a}=\key_{\a}(x_1, \dots, x_\ell)$ is defined as  
\begin{equation*}
\key_\a=\sum_{T\in \KT(\a)} x^{\wt(T)}\,,
\end{equation*}
\item the \emph{fundamental slide polynomial} $\fund_{\a}=\fund_{\a}(x_1, \dots, x_\ell)$ is defined as 
\begin{equation*} 
\fund_{\a}=\sum_{\stackrel{\b\geq \a}{ \mathrm{flat}(\b) \text{ refines } \mathrm{flat}(\a)}} x^\b\,. 
\end{equation*}
\end{enumerate}
\end{definition}

\svw{If we let $s_\lambda$ denote the Schur function indexed by the partition $\lambda$ and $F_\alpha$ denote the fundamental quasisymmetric function indexed by the strong compsition $\alpha$, then we have the following.}
\begin{theorem}[\cite{AS18}] \label{thm:limit} 
For a weak composition $\a$, we have
\begin{eqnarray}
     \lim_{m\rightarrow \infty} \key_{0^m\times \a}=s_{\mathrm{sort}(\a)}\,,\\
     \lim_{m\rightarrow \infty} \fund_{0^m\times \a}=F_{\mathrm{flat}(\a)}\,,
\end{eqnarray}\svw{where $0^m\times \a$ means prepending $m$ $0$s to $\a$.}
\end{theorem}
%We take the following as the definition of \emph{key polynomials} while there are many ways to define them.

\begin{proposition}[Theorem 2.13, \cite{AS18}] \label{prop:key2slide}
For a weak composition $\a$, the \emph{key polynomial indexed by $\a$} is $$\key_\a=\sum_{T\in \QKT(\a)} \fund_{\wt(T)}\,.$$
\end{proposition}

\begin{example}\label{ex:keyasfund} In Figure~\ref{fig:QKT(0032)}, all quasi-Yamanouchi Kohnert tableaux of content $(0,0,3,2)$ are given, and 
due to Proposition~\ref{prop:key2slide} we have
$$\key_{(0,0,3,2)}=\fund_{(0,0,3,2)}+ \fund_{(0,2,2,1)}+\fund_{(0,1,3,1)}+\fund_{(0,2,3,0)}+\fund_{(1,2,2,0)}\,.$$ 
\end{example}
%%%%%%%%%%%%%%%%%%%%%%%%%%%%%%%%%%%
\section{Initial results} \label{sec:initial}

We now prove some initial multiplicity free classifications, which will be useful later. In particular we classify when  a key polynomial is equal to a fundamental slide polynomial or the sum of two fundamental slide polynomials. In each case we give the precise decomposition.

%%%%%%%
\subsection{Classifying when $\key_\a=\fund_\a$}\label{subsec:oneterm}

Before establishing our classification we prove \svwii{four} lemmas.

\begin{lemma} \label{lem:1}
Let $\a=(a_1, \dots, a_k, \underbrace{0,\dots, 0}_m, a_{k+m+1})$ for positive integers $a_1\geq\cdots\geq a_k$,  a nonnegative integer $a_{k+m+1}\leq a_k$ and a nonnegative integer $m$. Then $\key_\a=\fund_\a$.
\end{lemma}
\begin{proof}
The basic quasi-Yamanouchi Kohnert tableau of content $\a$ is the unique element of  $\QKT(\a)$. Due to Proposition~\ref{prop:key2slide}, we can conclude that $\key_\a=\fund_\a$.
\end{proof}

\svwM{\begin{lemma} \label{lem:1point5}
Let $\a=(a_1, \dots, a_k, \overline{a}_{01}, 1)$ for positive integers $a_1\geq\cdots\geq a_k>1$   and $\overline{a}_{01}$ a finite sequence of $0$s and $1$s. Then $\key_\a=\fund_\a$.
\end{lemma}
\begin{proof}
The basic quasi-Yamanouchi Kohnert tableau of content $\a$ is the unique element of  $\QKT(\a)$. Due to Proposition~\ref{prop:key2slide}, we can conclude that $\key_\a=\fund_\a$.
\end{proof}}

\begin{lemma} \label{lem:2}
If a weak composition $\a=(a_1,   \dots, a_\ell)$  has two nonzero entries $0<a_i<a_j$ with $i<j$, then the key polynomial $\key_\a$  has at least two different terms in its expansion into  fundamental slide polynomials.  
\end{lemma}
\begin{proof}  Choose $i<j$ with $0<a_i<a_j$ so that $(j-i)$ is minimal. Then $a_{i+1}=\cdots=a_{j-1}=0$, and moving the last cell of row $j$ of the basic quasi-Yamanouchi Kohnert tableau down to row $i$ produces another tableau in $\QKT(\a)$.  \end{proof}

\begin{lemma} \label{lem:3}
Suppose that a weak composition $\a=(a_1,   \dots, a_\ell)$  has two nonzero entries $a_i, a_j$  with $i<j$ and at least one of $a_i$ and $a_j$ is strictly bigger than $1$. If there is $h<i$ such that $a_h=0$,  then the key polynomial $\key_\a$ has at least two different terms in its expansion into  fundamental slide polynomials.  Moreover, if $a_i\neq 1$ then one term $\fund_{(b_1, \ldots , b_\ell)}$ has $b_h\neq 0$.
\end{lemma}

\begin{proof} Let $h$ be the smallest index such that $a_h=0$ but $a_{h+1}\not=0$, and let $i=h+1$.  Then, by assumption we  can find the \emph{smallest}  $j>h+1$ with $a_j>0$ so that at least one of  $a_i, a_j$  is strictly bigger than $1$.  There are three cases to consider.

If \svw{$1=a_i< a_j$,} then the result follows by Lemma~\ref{lem:2}.

If $1<a_i\leq a_j$, then $a_k=0$ for all $k=i+1, \dots, j-1$. Therefore, moving the last $i$ in row $i$ down to row $i-1=h$ and the last $j-i+1$ entries in row $j$ down to row $i$ from the basic quasi-Yamanouchi Kohnert tableau will produce another tableau in $\QKT(\a)$. See Figure~\ref{fig:ai=aj}. %, since there must be at least one $i$ in row $i$ of the resulting tableau.  See Figure~\ref{fig:ai=aj}.

\begin{figure}[ht]
\begin{center}
    \begin{tikzpicture}[xscale=2,yscale=2] 
       \node at (0,0)  {$ \vline\tableau{\\j&\cdots &j&j&\cdots&j\\ \\ i&\cdots &i&i\\ \\ \\ \hline} $};
     \node at (2,0)  {$ \vline\tableau{\\j&\cdots &j\\ \\ i&\cdots &i&j&\cdots&j\\ & & &i\\ \\ \hline} $};
     \node at (-0.95,0.43) {\tiny$\mathrm j$};\node at (-0.95,-0.14) {\tiny$\mathrm i$};\node at (-0.95,-0.43) {\tiny$\mathrm h$};
\node at (1.05,0.43) {\tiny$\mathrm j$};\node at (1.05,-0.14) {\tiny$\mathrm i$};\node at (1.05,-0.43) {\tiny$\mathrm h$};
\end{tikzpicture}
  \end{center}
  \caption{\label{fig:ai=aj} Two quasi-Yamanouchi Kohnert tableaux when $1<a_i\leq a_j$ and $a_h=0$.}
\end{figure}
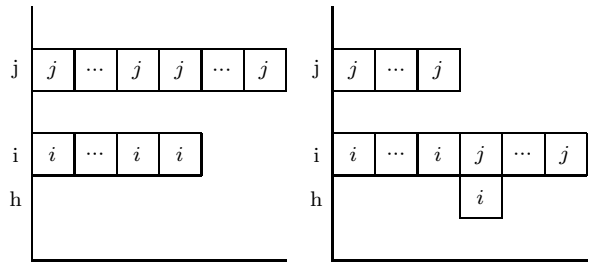
 
If \svw{$a_i>a_j\geq1$,} then  $a_k=0$ for all $k=i+1, \dots, j-1$. Therefore, moving the last $j$ in row $j$ down to row $i-1=h$ from the basic quasi-Yamanouch Kohnert tableau will produce another tableau in $\QKT(\a)$.  See Figure~\ref{fig:ai>aj}.
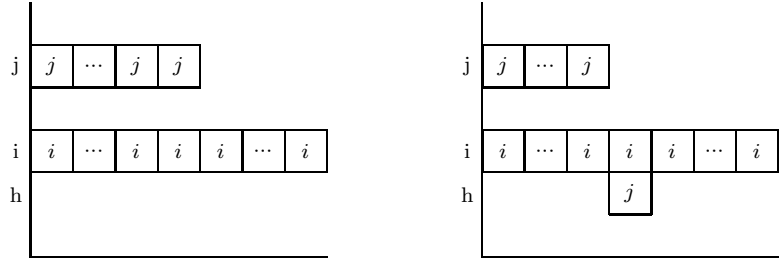
\begin{figure}[ht]
\begin{center}
    \begin{tikzpicture}[xscale=3,yscale=2] 
       \node at (0,0)  {$ \vline\tableau{\\j&\cdots &j&j\\ \\ i&\cdots &i&i&i&\cdots&i\\ \\ \\ \hline} $};
     \node at (2,0)  {$ \vline\tableau{\\j&\cdots &j\\ \\ i&\cdots &i&i&i&\cdots&i\\  & &  & j & & \\ \\ \hline} $};
     \node at (-0.72,0.43) {\tiny$\mathrm j$};\node at (-0.72,-0.14) {\tiny$\mathrm i$};\node at (-0.72,-0.43) {\tiny$\mathrm h$};
\node at (1.28,0.43) {\tiny$\mathrm j$};\node at (1.28,-0.14) {\tiny$\mathrm i$};\node at (1.28,-0.43) {\tiny$\mathrm h$};
\end{tikzpicture}
  \end{center}
  \caption{\label{fig:ai>aj} Two quasi-Yamanouchi Kohnert tableaux when $a_i>a_j\geq 1$ and $a_h=0$.}
\end{figure}

Finally note that in both the $a_i\neq 1$ cases row $h$ has a cell in it.
\end{proof}

\begin{theorem}\label{thm:k=f}
For a weak composition $\a$, the key polynomial $\key_\a$ has a unique term in the expansion into fundamental slide polynomials, that is $\key_\a=\fund_\a$, if and only if $\a$ is one of the following weak compositions.
\begin{enumerate}
\item Every  nonzero part of $\a$ is $1$.
\item $\a$ has only one nonzero part.
\item   $\a=(a_1, \dots, a_k, \underbrace{0,\dots, 0}_m, a_{k+m+1})$ for positive integers $a_1\geq\cdots\geq a_k$, a nonnegative integer $a_{k+m+1}\leq a_k$ and a nonnegative integer $m$.
\item \svwM{$\a=(a_1, \dots, a_k, \overline{a}_{01}, 1)$ for positive integers $a_1\geq\cdots\geq a_k>1$   and $\overline{a}_{01}$ a finite sequence of $0$s and $1$s.}
\end{enumerate}
\end{theorem}
\begin{proof} It is easy to check that if either (1) or (2) is the case, then $\key_\a=\fund_\a$, and Lemma~\ref{lem:1},  \svwM{Lemma~\ref{lem:1point5},} Lemma~\ref{lem:2} and  Lemma~\ref{lem:3} complete the proof.
\end{proof}

\begin{example}\label{ex:k=f} If $\a = (3,0,0,2)$, then $\key _{(3,0,0,2)}=\fund _{(3,0,0,2)}$.
\end{example}

\begin{corollary}\label{cor:one_term}
For a strong composition $\alpha$, $\key_\alpha=\fund_\alpha$ if and only if $\alpha$ is a partition.
\end{corollary}

%%%%%%%
\subsection{Classifying when $\key_\a=\fund_\a+\fund_\b$}\label{subsec:twoterm}

We will first deal with the case of a strong composition and then apply \svwii{this case} to obtain the full result. We begin with two lemmas.

\begin{lemma}\label{lem:two}
For a strong composition $\alpha$, $\fund_\alpha$ and $\fund_{\mathrm{sort}(\alpha)}$ are always \svw{(not necessarily distinct)} terms in $\key_\alpha$.
\end{lemma}
\begin{proof} $\fund_\alpha$ comes from the basic \svwii{quasi-Yamanouchi} Kohnert tableau. $\fund_{\mathrm{sort}(\alpha)}$ comes from the \svwii{quasi-Yamanouchi} Kohnert tableau obtained from the basic one by bottom justifying each column with its entries written in increasing order from bottom to top.
\svwii{If $\alpha = \mathrm{sort}(\alpha)$,} then these two \svwii{quasi-Yamanouchi} Kohnert tableaux are the same tableau.\end{proof}

\begin{definition} \label{def:inv}For a strong composition $\alpha=(\alpha_1,   \dots, \alpha_\ell)$, a pair $(i,  j)$,  $i<j$, is an \emph{inversion} of $\alpha$ if $0< \alpha_i<\alpha_j$, and we let $\mathrm{inv}(\alpha)$ be the number of inversions of $\alpha$.
\end{definition}

\begin{example}\label{ex:inv} If $\alpha = (2,3)$, then $\mathrm{inv}(\alpha) = 1$.
\end{example}

\begin{lemma}\label{lem:expansion}
If $\mathrm{inv}(\alpha)=1$, then the unique inversion must be the pair $(i, i+1)$ for some $i$. If $\alpha=(\alpha_1, \dots, \alpha_{i-1}, \alpha_i, \alpha_i+m, \alpha_{i+2}, \dots, \alpha_\ell)$ where  $m>0$ and $\alpha_1\geq \cdots \geq \alpha_i \geq \alpha_{i+2}\geq\cdots \geq \alpha_\ell$, is a strong composition with only one inversion, then 
\[\key_\alpha=\sum_{t=0}^m \fund_{(\alpha_1,  \dots,  \alpha_{i-1}, \alpha_i+t,  \alpha_i+m-t,  \alpha_{i+2}, \dots,  \alpha_\ell)}\,.\]
\end{lemma}
\begin{proof} 
If $\alpha_i\geq \alpha_{i+1}$ for all $i$, then $\mathrm{inv}(\alpha)=0$. Hence, if $\mathrm{inv}(\alpha)=1$ then $\alpha_i<\alpha_{i+1}$ must hold for some $i$ and this must be the unique such $i$, plus $\alpha_{i+1}\leq \alpha_k$ for all $1\leq k\leq i-1$.

All Kohnert tableaux are the ones obtained by moving cells in row $(i+1)$ down to row $i$ one by one from the rightmost cell and these are all quasi-Yamanouchi, which proves the given equation.
\end{proof}

We remark that the weight of $T\in \QKT(\alpha)$ for a strong composition $\alpha$, is a strong composition since row $i$ of the first column of any Kohnert tableau of content $\alpha$ is filled with $i$. Hence we know that if $\fund_\a$ appears in the expansion of $\key_\alpha$, then $\a$ must be a strong composition of the same length as $\alpha$.

\begin{proposition} \label{prop:2termsstrong} For a strong composition $\alpha$, $\key_\alpha=\fund_\alpha+\fund_\beta$ if and only if $\mathrm{inv}(\alpha)=1$ and $\alpha_{i+1}=\alpha_i+1$ for the unique inversion $(i, i+1)$ of $\alpha$. Furthermore, $\beta=\mathrm{sort}(\alpha)$ in this case.
\end{proposition}

\begin{proof}
Suppose that $\key_\alpha=\fund_\alpha+\fund_\beta$ for some strong composition $\beta$. Then, by  Lemma~\ref{lem:two}, $\beta$ must be $\mathrm{sort}(\alpha)$ and $\mathrm{inv}(\alpha)>0$. Let $i$ be the smallest index such that $(i, i+1)$ is an inversion of $\alpha$. Then $\alpha_{i+1}=\alpha_i+1$ must hold, since otherwise at least three terms will appear in the expansion of $\key_\alpha$ into fundamental slide polynomials by moving cells in row $i+1$ down to row $i$ one by one from the rightmost cell. Now, if $\alpha_{i+2}>\alpha_i$ then either $\alpha_{i+2}>\alpha_{i+1}$ or $\alpha_{i+2}=\alpha_{i+1}$ must hold and in either case, by moving cells in row $i+1$ down to row $i$ and then moving cells in row $i+2$ down to row $i+1$, there are at least three terms in the expansion of $\key_\alpha$. Hence we have $\alpha_{i+2}\leq \alpha_i$.  \svw{Since if there is $j\geq i+2$ such that $(j, j+1)$ is an inversion of $\alpha$, then this will make at least two more terms in the expansion of  $\key_\alpha$,  we} can conclude that the only inversion of $\alpha$ is $(i, i+1)$ and $\alpha_{i+1}=\alpha_i+1$.

The proof of the other direction is immediate from Lemma~\ref{lem:expansion}.
\end{proof}

\begin{example} \label{ex:2termsstrong} If $\alpha = (2,3)$, then $\key _{(2,3)} = \fund _{(2,3)} + \fund _{(3,2)}$.
\end{example}

We now apply our classification for strong compositions to obtain our full classification. We also need the following generalization of $\mathrm{sort}(\a)$. For a weak composition $\a$ we let $\mathrm{sort}_0(\a)$ be the weak composition whose nonzero parts are the parts of $\mathrm{sort}(\a)$ taken in weakly decreasing order, and the $i$th part of $\mathrm{sort}_0(\a)$ is 0 if and only if the $i$th part of $\a$ is 0. For example, if $\a = (0,2,0,3)$ then $\mathrm{sort}_0(\a)= (0,3,0,2)$.

\begin{theorem} \label{thm:2terms} For a weak composition $\a$,  $\key_\a=\fund_\a+\fund_\b$ if and only if $\a$ satisfies both the following conditions:
\begin{enumerate}
\item[(a)] $\mathrm{inv}(\mathrm{flat}(\a))=1$ and $\mathrm{flat}(\a)_{i+1}=\mathrm{flat}(\a)_i+1$ for the unique inversion $(i, i+1)$ of $\mathrm{flat}(\a)$, 
\item[(b)] \svwM{either 
\begin{itemize}
\item $\mathrm{flat}(\a) = (1,2)$, or 
\item $\a=(a_1, \dots, a_k, \underbrace{0,\dots, 0}_m, a_{k+m+1})$ for positive integers \svwM{$a_1,\ldots ,a_k$,} a nonnegative integer $a_{k+m+1}$ and a nonnegative integer $m$, or 
\item $\a=(a_1, \dots, a_k, \overline{a}_{01}, 1)$ for positive integers $a_1, \ldots, a_k>1$   and $\overline{a}_{01}$ a finite sequence of $0$s and $1$s.\end{itemize}}
\end{enumerate}
Furthermore, $\b=\mathrm{sort}_0(\a)$ in this case.
\end{theorem}

\begin{proof} Let $\alpha$ be a strong composition with $\key_\alpha\neq\fund_\alpha+\fund_\beta$ and $\a$ be a weak composition such that $\mathrm{flat}(\a)=\alpha$. Then $\key_\a\neq\fund_\a +\fund_\b$. This is because if $|\QKT(\alpha)| >2$ then by adding empty rows and increasing the numbers in the cells by the number of empty rows inserted below we will produce more than two tableaux in $\QKT(\a)$.

Hence $\mathrm{flat}(\a)$ must satisfy the conditions of Proposition~\ref{prop:2termsstrong}, giving us the first condition. By Lemma~\ref{lem:two} and noting that we can produce members in $\QKT(\a)$ from $\QKT(\alpha)$ by adding empty rows and increasing the numbers in the cells by the number of empty rows inserted below, we know that $\key_\a=\fund_\a+\fund_{\mathrm{sort}_0(\a)}$ plus perhaps other terms. \svwM{To complete the proof note that Lemma~\ref{lem:3} then guarantees that we will have other terms unless $\mathrm{flat}(\a) = (1,2)$, or $\a=(a_1, \dots, a_k, \underbrace{0,\dots, 0}_m, a_{k+m+1})$ for positive integers \svwM{$a_1,\ldots ,a_k$,} a nonnegative integer $a_{k+m+1}$ and a nonnegative integer $m$, or $\a=(a_1, \dots, a_k, \overline{a}_{01}, 1)$ for positive integers $a_1,\ldots, a_k>1$   and $\overline{a}_{01}$ a finite sequence of $0$s and $1$s.}
 In this case, by the definition of quasi-Yamanouchi Kohnert tableaux and Proposition~\ref{prop:key2slide} no further terms are produced, giving us the second condition.
\end{proof}

\begin{example}\label{ex:2terms} If $\alpha = (2,0,0,3)$, then $\key _{(2,0,0,3)} = \fund _{(2,0,0,3)} + \fund _{(3,0,0,2)}$.
\end{example}

%%%%%%%%%%%%%%%%%%%%%%%%%%%%%%%%%%%%%%%%%%%%%%%%%%%%%%%%%%%%%%
\section{Classifying when  $\key_\alpha$ is multiplicity free for $\alpha$ a strong composition}\label{sec:MFforStrong}

In this section, we restrict our attention to strong compositions.

%\begin{theorem}\label{thm:main1}
%Let $\alpha=(\alpha_1, \dots, \alpha_l)$ be a strong composition.
%\begin{enumerate}
%\item If there exist $1<i<j$ such that $\alpha_i<\alpha_j$ and $\alpha_1, \dots, \alpha_{i-1}, \alpha_{i+1},\dots, \alpha_{j-1}<\alpha_i$ then  $\key_\alpha$ is \emph{not} multiplicity free.
%\item If there exist $2<i<j$ such that $\alpha_i>\alpha_j$ and $\alpha_1, \dots, \alpha_{i-1}, \alpha_{i+1},\dots, \alpha_{j-1}<\alpha_j$ then  $\key_\alpha$ is \emph{not} multiplicity free.
%\item If there exist $2<i<j$ such that $\alpha_i=\alpha_j$ and $\alpha_1, \dots, \alpha_{i-1}, \alpha_{i+1},\dots, \alpha_{j-1}<\alpha_i-1$ then  $\key_\alpha$ is \emph{not} multiplicity free.
%\end{enumerate}
%\end{theorem}

\begin{theorem}\label{thm:main1}
Let $\alpha=(\alpha_1, \dots, \alpha_\ell)$ be a strong composition.
\begin{enumerate}
\item If there exist $i<j<k$ such that $\alpha_i<\alpha_j<\alpha_k$, then  $\key_\alpha$ is \emph{not} multiplicity free.
\item If there exist $i<j<k<l$ such that $\alpha_i, \alpha_j<\alpha_l<\alpha_k$, then  $\key_\alpha$ is \emph{not} multiplicity free.
\item If there exist $i<j<k<l$ such that $\alpha_i, \alpha_j+1<\alpha_k=\alpha_l$, then  $\key_\alpha$ is \emph{not} multiplicity free.
\end{enumerate}
\end{theorem}

\begin{example} \label{ex:main1} The key polynomials $\key _{(1,2,3)}$, $\key _{(1,1,3,2)}$ and $\key _{(1,1,3,3)}$ are \emph{not} multiplicity free.
\end{example}

\begin{proof}
Suppose that there are $i<j<k$ such that $\alpha_i<\alpha_j<\alpha_k$. We may assume that $k>j$ is the smallest $k$ such that $\alpha_j<\alpha_k$ and $i<j$ is the largest $i$ such that $\alpha_i<\alpha_j$, which imply that $\alpha_x\leq\alpha_j$ for all $j<x<k$ and $\alpha_y\geq \alpha_j$ for all $i<y<j$. The basic tableau of weight $\alpha$ is the first one in Figure~\ref{fig:(1)}, in which $A$ is a subtableau contained in the first $\alpha_j$ columns. The second and the third tableaux of Figure~\ref{fig:(1)} are quasi-Yamanouchi Kohnert tableaux of the same weight with content $\alpha$ and we can conclude that $\key_\alpha$ is not multiplicity free. We remark that the row filled with $*$s in the first tableau of Figure~\ref{fig:(1)} may \svwii{not exist} in which case the \svwii{bottom} two rows in the second and the third tableaux will \svw{be a single} row without $*$s. 

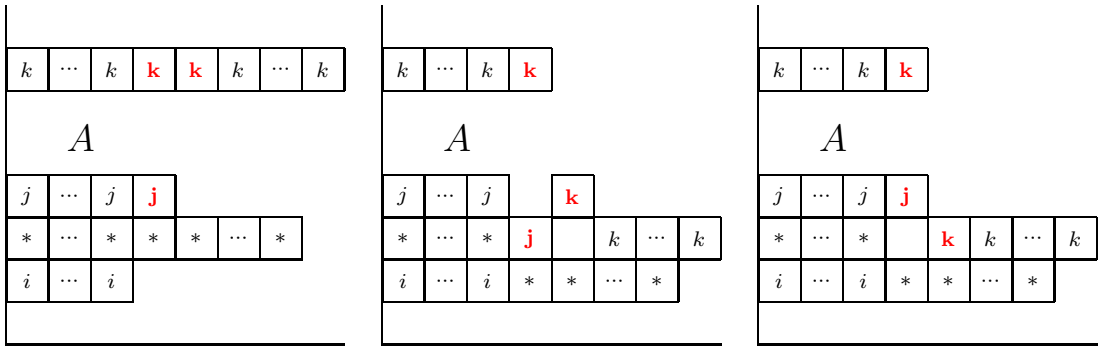
\begin{figure}[ht]
\begin{center}
    \begin{tikzpicture}[xscale=2.5,yscale=2] 
       \node at (0,0)  {$ \vline\tableau{\\k&\cdots &k&{\red\mathbf k}&{\red\mathbf k}&k&\cdots&k \\ \\  \\j&\cdots &j&{\red\mathbf j}\\ * & \cdots & * & * &*&\cdots& * \\ i&\cdots &i \\ \\\hline} $};
     \node at (2,0)  {$ \vline\tableau{\\k&\cdots &k& {\red\mathbf k} \\ \\  \\j&\cdots &j&&{\red\mathbf k}\\ *&\cdots&*&{\red\mathbf j}&&k&\cdots&k\\ i&\cdots &i&  * &*&\cdots& * \\ \\ \hline} $};
     \node at (4,0)  {$ \vline\tableau{\\k&\cdots &k& {\red\mathbf k} \\ \\  \\j&\cdots &j&{\red\mathbf j}\\*&\cdots&*&&{\red\mathbf k}&k&\cdots&k\\i&\cdots &i&  * &*&\cdots& * \\ \\ \hline} $};
     \node at (-0.5, 0.25) {\large $A$};\node at (1.5, 0.25) {\large $A$};\node at (3.5, 0.25) {\large $A$};
    % \node at (-0.72,0.43) {\tiny$\mathrm j$};\node at (-0.72,-0.14) {\tiny$\mathrm i$};
     %\node at (1.28,0.43) {\tiny$\mathrm j$};\node at (1.28,-0.14) {\tiny$\mathrm i$};
     %\node at (3.28,0.43) {\tiny$\mathrm j$};\node at (3.28,-0.14) {\tiny$\mathrm i$};
\end{tikzpicture}
  \end{center}
  \caption{\label{fig:(1)}  Basic tableau and two tableaux of same weight in $\QKT(\alpha)$ when $\alpha$ satisfies (1).}
\end{figure}

\svw{Now we will make a remark, which will be used often throughout the remainder of the proof, and will be referred to as ``the remark above.''}
We \textbf{remark} that if there is a row that is strictly longer than the rows we consider in a Kohnert tableau, then we can move cells across that row without violating any condition for being a quasi-Yamanouchi Kohnert tableau. We thus do not take account of the rows that are strictly longer than the rows we are considering.

Suppose that there are $i<j<k<l$ such that $\alpha_i, \alpha_j<\alpha_l<\alpha_k$. If there is $x>k$ such that $\alpha_x>\alpha_k$, then we have $\alpha_j<\alpha_k<\alpha_x$ and Case (1) is satisfied with $j<k<x$ and hence $\key_\alpha$ is not multiplicity free. Hence, we may assume that $\alpha_x\leq \alpha_k$ for all $x>k$, and $l>k$ is the smallest $l$ such that $\alpha_l<\alpha_k$. Then, $\alpha_x=\alpha_k$ for all $k<x<l$. On the other hand, if $\alpha_i<\alpha_y<\alpha_k$ for any $i<y<k$, then Case (1) is satisfied and hence $\key_\alpha$ is not multiplicity free. Therefore, for all $i<y<k$, $\alpha_y\leq \alpha_i$ so $\alpha_j\leq \alpha_i$ and moreover we can assume that $j=k-1$ and $i=k-2$. %Figure~\ref{fig:(2)} shows for $\alpha_i=\alph_j$ that there are at least two quasi-Yamanouch Kohnert tableaux with same weight and hence $\key_\alpha$ is not multiplicity free. 
The basic tableau of weight $\alpha$ is the first one in Figure~\ref{fig:(2)} for $\alpha_i=\alpha_j=\alpha_\ell - 1$. The second and the third tableaux of Figure~\ref{fig:(2)} are quasi-Yamanouchi Kohnert tableaux of the same weight with content $\alpha$ and we can conclude that $\key_\alpha$ is not multiplicity free. The figures for $\alpha_j<\alpha_i$, and other suitable $\alpha _\ell$, are almost identical.

\begin{figure}[ht]
\begin{center}
    \begin{tikzpicture}[xscale=2.5,yscale=2] 
       \node at (0,0) 
        {$ \vline\tableau
           {\\ l&\cdots &l&{\red\mathbf l}
            \\ k+t &\cdots &k+t &{\red\mathbf {k+t}}&{\red\mathbf{ k+t}}&k+t&\cdots&k+t 
            \\  \\ \\k+1&\cdots &k+1&{\red\mathbf {k+1}}&{\red\mathbf {k+1}}&k+1&\cdots&k+1
             \\k&\cdots &k&{\red\mathbf k}&{\red\mathbf k}&k&\cdots&k
            \\ j&\cdots&j  \\  i & \cdots & i   \\     \\  \hline} $ };
     \node at (2,0)  
      {$ \vline\tableau
        {\\  l&\cdots &l
         \\ k+t & \cdots & k+t & {\red\mathbf l}
         \\ & & &{\red\mathbf {k+t}}&{\red\mathbf {k+t}} 
         \\ & & & & & k+t&\cdots &k+t \\  k+1&\cdots&k+1
         \\k&\cdots &k&{\red\mathbf {k+1}}&{\red\mathbf {k+1}}
         \\j&\cdots&j&{\red\mathbf k}&&k+1&\cdots&k+1
         \\i&\cdots&i&&{\red\mathbf k}&k&\cdots&k
         \\ \\  \hline} $};
     \node at (4,0)  
     {$ \vline\tableau
        {\\  l&\cdots &l
         \\ k+t & \cdots & k+t & {\red\mathbf {k+t}}
         \\ & & &&{\red\mathbf {k+t}} 
         \\ & & & & & k+t&\cdots &k+t \\  k+1&\cdots&k+1&{\red\mathbf {k+1}}
         \\k&\cdots &k&{\red\mathbf {k}}&{\red\mathbf {k+1}}
         \\j&\cdots&j&&{\red\mathbf k}&k+1&\cdots&k+1
         \\i&\cdots&i&{\red\mathbf l}&&k&\cdots&k
         \\ \\  \hline} $};
     \node at (0, 0.25){ $\cdots$};  
     \node at (2, 0){ $\cdots$}; \node at (1.55, 0.25){ $\cdots$}; \node at (2.45, -0.26){ $\cdots$}; 
     \node at (4.1, 0){ $\vdots$}; \node at (3.65, 0.25){ $\cdots$}; \node at (4.5, -0.26){ $\cdots$}; 
    % \node at (-0.72,0.43) {\tiny$\mathrm j$};\node at (-0.72,-0.14) {\tiny$\mathrm i$};
     %\node at (1.28,0.43) {\tiny$\mathrm j$};\node at (1.28,-0.14) {\tiny$\mathrm i$};
     %\node at (3.28,0.43) {\tiny$\mathrm j$};\node at (3.28,-0.14) {\tiny$\mathrm i$};
\end{tikzpicture}
  \end{center}
  \caption{\label{fig:(2)} Basic tableau and two tableaux of same weight in $\QKT(\alpha)$ when $\alpha$ satisfies (2). }
\end{figure}
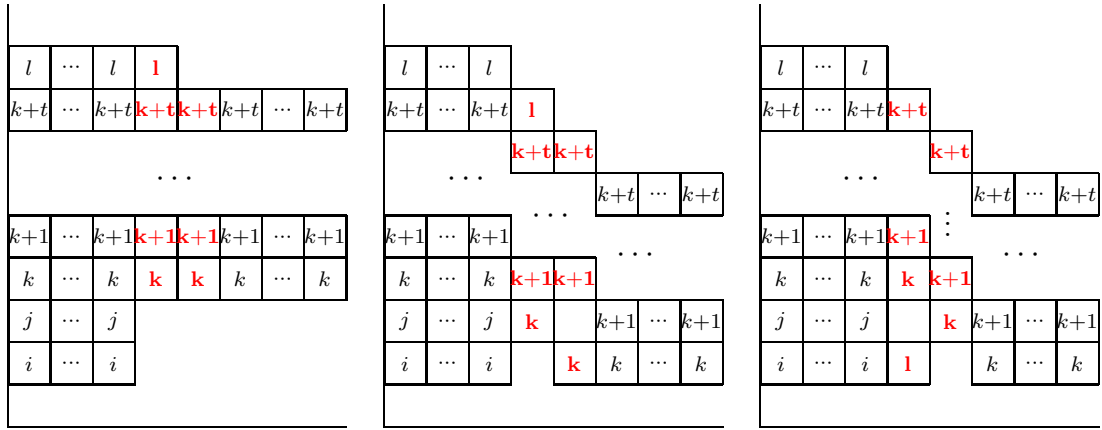

Suppose that there are $i<j<k<l$ such that $\alpha _i, \alpha _j +1 < \alpha _k = \alpha _l$. We choose the smallest $k$ and largest $j$ so that there is no $y$ such that $j<y<k$ and $\alpha _y = \alpha _k$ or $\alpha _y = \alpha _j$. We may assume there is also no $y$ such that $j<y<k$ and $\alpha_j<\alpha_y<\alpha_k$ since this is Case (1). Thus by the remark above we may assume that $\alpha_y<\alpha _j$ for all $j<y<k$. In Figure~\ref{fig:(3)} we will not draw these rows, which we will fix.

Now we choose the smallest $l$ so that there is no $x$ such that $k<x<l$ and $\alpha _x=\alpha _l$. Then by the remark above we may assume that for all $k<x<l$, $\alpha _x< \alpha _l$. Furthermore we may assume $\alpha _x \leq \alpha _j$ for all $k<x<l$ to ensure there is no $x$ such that $j<x<l$ and $\alpha_j<\alpha_x<\alpha_l$, since this is Case (1). Hence
$$\alpha _x \leq \alpha _j < \alpha _k -1 \Rightarrow \alpha _x \leq \alpha _k -2.$$ In Figure~\ref{fig:(3)} we will not draw these rows, which we will fix.
Now we choose the largest $i$ so there is no $z$ such that $i<z<j$ and  $\alpha_z<\alpha_k$.  Furthermore we can assume that $\alpha _z \geq \alpha _k$ for all $i<z<j$ as otherwise we can choose $z<j<k<l$ such that $\alpha _z, \alpha _j +1 < \alpha _k = \alpha _l$. Thus by the remark above we can assume that $\alpha _z = \alpha _k$. In Figure~\ref{fig:(3)}, we will denote the rows $i<z<j$ where $\alpha _z = \alpha _k$ collectively by $\ast$ and $\star$. \svw{We note that our construction also works if $j=i+1$.}

We now consider the various cases for $\alpha _i$ and $\alpha _j$. We do not need to consider the case $\alpha _i < \alpha _j$ since this is Case (1). For the case $ \alpha _i = \alpha_j +1$ we have that the basic tableau of weight $\alpha$ is the first one in Figure~\ref{fig:(3)}. The second and third tableaux of Figure~\ref{fig:(3)} are quasi-Yamanouchi Kohnert tableaux of the same weight with content $\alpha$ and we can conclude in this case that $\kappa _\alpha$ is not multiplicity free. The figures for the cases $\alpha _i = \alpha _j$ and $ \alpha _i > \alpha _j +1 $ are almost identical. 
%Suppose that there are $i<j<k<l$ such that $\alpha_i, \alpha_j+1<\alpha_k=\alpha_l$. Choose the smallest $k$ satisfying the conditions so there is no $x$ such that $j<x<k$ and $\alpha_x=\alpha_k$. There is also no $x$ such that $j<x<k$ and $\alpha_j<\alpha_x<\alpha_k$ since this is Case (1) and hence $\key_\alpha$ is not multiplicity free. Similarly, we can assume that $\alpha_i=\alpha_j+1$ or $\alpha_j$.
%Choosing the smallest $l$ satisfying the conditions we may assume $\alpha_y<\alpha_k$ for all $k<y<l$. In Figure~\ref{fig:(3)}, we will not draw these rows, which we will fix.
%We may also assume $\alpha_y\leq \alpha_j$ or $\alpha_y=\alpha_k$ for all $i<y<j$ otherwise there exists $j<y<k$ and $\alpha_j<\alpha_y<\alpha_k$, which is Case (1) and hence 
%$\key_\alpha$ is not multiplicity free. In Figure~\ref{fig:(3)}, we denote the rows $i<y<j$ where $\alpha_y=\alpha_k$ collectively by $*$ and $\star$ and do not draw the rows $\alpha_y\leq \alpha_j$ which we will fix.
%The basic tableau of weight $\alpha$ is the first one in Figure~\ref{fig:(3)} for $\alpha_j+1=\alpha_i$. The second and the third tableaux of Figure~\ref{fig:(2)} are quasi-Yamanouchi Kohnert tableaux of a same weight with content $\alpha$ and we can conclude that $\key_\alpha$ is not multiplicity free. The figures for $\alpha_j=\alpha_i$ are almost identical.
% By the same argument, we can assume that $\alpha_y$ for $i<y<j$ and for $k<y<l$. We can check that there are at least two quasi-Yamanouchi Kohnert tableaux of a same weight. See Figure~\ref{fig:(3)}.
\begin{figure}[ht]
\begin{center}
    \begin{tikzpicture}[xscale=2.5,yscale=2] 
       \node at (0,0)  
       {$ \vline\tableau{
        \\l&\cdots &l&\cdots&\cdots&{\red\mathbf l}&{\red\mathbf l}
        \\  \\k&\cdots &k&\cdots&\cdots&{\red\mathbf k}&{\red\mathbf k}
        \\ 
        \\ j & \cdots & j
        \\  * & \cdots & *&\cdots&\cdots & * &*
        \\ \star & \cdots & \star&\cdots &\cdots& \star &\star
        \\ i & \cdots & i&i  \\ \\ \hline} $};
     \node at (2,0)  
     {$ \vline\tableau{
     \\l&\cdots &l&\cdots&\cdots&{\red\mathbf l}
     \\  \\k&\cdots &k&\cdots&\cdots&{\red\mathbf k}
     \\ \\j & \cdots & j& &&&{\red\mathbf l} 
     \\   * & \cdots & *&\cdots &\cdots& * &{\red\mathbf k}
     \\   \star & \cdots & \star&\cdots &\cdots& \star & *
     \\i & \cdots & i &i & & & \star\\ \\ \hline} $};
     \node at (4,0)   
     {$ \vline\tableau{
     \\l&\cdots &l&\cdots&\cdots&{\red\mathbf l}\\ 
     \\k&\cdots &k&\cdots&\cdots&&{\red\mathbf l}\\
     \\j&\cdots&j&& &{\red\mathbf k} 
     \\*&\dots&*&\cdots&\cdots&*&{\red\mathbf k} 
     \\   \star & \cdots & \star&\cdots &\cdots& \star & *
     \\ i&\cdots&i&i&&&\star\\ \\ \hline} $};
    % \node at (-0.72,0.43) {\tiny$\mathrm j$};\node at (-0.72,-0.14) {\tiny$\mathrm i$};
     %\node at (1.28,0.43) {\tiny$\mathrm j$};\node at (1.28,-0.14) {\tiny$\mathrm i$};
     %\node at (3.28,0.43) {\tiny$\mathrm j$};\node at (3.28,-0.14) {\tiny$\mathrm i$};
\end{tikzpicture}
  \end{center}
  \caption{\label{fig:(3)} Basic tableau and two tableaux of same weight in $\QKT(\alpha)$ when $\alpha$ satisfies (3). }
\end{figure}
\end{proof}

We are now ready to give our classification of when $\key_\alpha$, for $\alpha$ a strong composition, is \svw{multiplicity free, and see that the conditions in Theorem~\ref{thm:main1} classify when $\key_\alpha$ has multiplicities.}

%\bigskip
\begin{theorem}\label{thm:main2} $\key_\alpha$ is multiplicity free if and only if $\alpha$ satisfies  all three of the following conditions:
\begin{enumerate}
\item[(a)] There is no $i<j<k$ such that $\alpha_i<\alpha_j<\alpha_k$.
\item[(b)] There is no $i<j<k<l$ such that $\alpha_i, \alpha_j<\alpha_l<\alpha_k$.
\item[(c)] There is no $i<j<k<l$ such that $\alpha_i, \alpha_j+1<\alpha_k=\alpha_l$.
\end{enumerate}
\end{theorem}
\begin{proof}
We use an induction on the number $\ell(\alpha)$ of parts of $\alpha$.

If $\ell(\alpha)=1$ or $\ell(\alpha)=2$ and  $\mathrm{inv}(\alpha)=0$, then $\key_\alpha=\fund_\alpha$ by Corollary~\ref{cor:one_term}. If $\ell(\alpha)=2$ and  $\mathrm{inv}(\alpha)=1$ then by Lemma~\ref{lem:expansion} $\key_\alpha$ is multiplicity free. 

We assume that the statement is true for all strong compositions $\beta$ with $\ell(\beta) < \ell$ for an $\ell \geq 2$. Let $\alpha=(\alpha_1,   \dots, \alpha_\ell)$ be a strong composition with $\ell$ parts, that satisfies Conditions (a), (b) and (c). 
Since $\mathrm{inv}(\alpha)=0$ implies $\key_\alpha=\fund_\alpha$ by Corollary~\ref{cor:one_term}, we assume that $\mathrm{inv}(\alpha)>0$ and let $i$ be the smallest index such that $\alpha_i<\alpha_{i+1}$.
We let $T_1$ and $T_2$ be two quasi-Yamanouchi Kohnert tableaux in $\QKT(\alpha)$ of the same weight. Then we claim that the first row of $T_1$ and $T_2$ are identical. We deal with  two cases  $i=1$ and $i>1$ separately to show the claim. \\
\begin{itemize}
\item[$i=1$: ] 
\svw{We will break this part of the proof into three parts. In the first part we will deduce the generic structure of our basic tableau in $\QKT (\alpha)$ when $\alpha_1<\alpha_2$, resulting in Figure~\ref{fig:i=1}. In the second part we will make three crucial observations on producing tableaux in $\QKT (\alpha)$. In the third part we will apply these observations to prove our claim using a proof by contradiction.}\\

\textbf{Deducing the structure:} Note first that since $\alpha_1<\alpha_2$ and $\alpha$ satisfies Condition (a), $\alpha_j\leq \alpha_2$ for $j>2$. Moreover, if  there are $2\leq j<k$ such that $\alpha_j<\alpha_{k}$, then $\alpha_j$ must be at most $\alpha_1$, that is, $\alpha_j\leq \alpha_1$ due to Condition (a). 

{  
Let us suppose that  $\alpha_j \leq \alpha_1<\alpha_{k}, \alpha_{k+1}$ for $2\leq j<k$. Then because of Condition (a),  $\alpha_l \leq \alpha_k$ for all $l>k$ and $\alpha_{k+1}$ is at most $\alpha_k$. If $\alpha_{k+1} < \alpha_k$, then $1<j<k<k+1$ gives a counterexample of Condition (b) and this implies that $\alpha_{k+1}=\alpha_{k}$. If  $\alpha_1+1<\alpha_{k+1}=\alpha_{k}$, then $1<j<k<k+1$ is a counterexample of  Condition (c) and we can conclude that $\alpha_{k}=\alpha_{k+1}=\alpha_1+1$. We now suppose that $\alpha_j<\alpha_1$, then $1<j<k<k+1$ is again a counterexample of  Condition (c), hence $\alpha_j=\alpha_1$.  
}

\svw{Consequently the basic tableau looks like Figure~\ref{fig:i=1}, where} 
we let $\{n_1<n_2<\cdots<n_x\}$ be the set $\{\alpha_j\,|\, j>2, \alpha_1<\alpha_j\}$ and let $B_m=\{ j \,|\, \alpha_j=n_m\}$, $m=1, 2, \dots, x$, be the set of  rows in the basic \svwii{quasi-Yamanouchi} Kohnert tableau of content $\alpha$ having $n_m$ cells. \\

\svw{Note that} if $x=1$, then $\alpha_2=\alpha_3=\cdots=\alpha_\ell $ and it is easy to see that by the definition of quasi-Yamanouchi Kohnert tableaux that the claim that the first row of $T_1$ and $T_2$ are identical is true. \svw{Therefore, let us assume that $x>1$.}\\

If $x>1$, then $B_2\cup B_3\cup \cdots \cup B_x=\{2, 3, \dots, \mathrm{max}(B_2)\}$, that is, there is no row shorter than  $\alpha_1+1$ among rows $2, 3, \dots, \mathrm{max}(B_2)$  by the definition of $B_m$. % and if $\alpha_j\leq \alpha_1$, $j\not=1$, then $j>\mathrm{max}(B_2)$. 
The yellow {box} in Figure ~\ref{fig:i=1} shows {the rows with $\alpha_1$ cells.}
{There can be more than one pair of the yellow box consisting of rows of length $\alpha_1$ and $B_1$ consisting of rows of length  $\alpha_1+1$. \svw{Above that there can be} a block that satisfies Conditions (a), (b), (c), of rows of length at most $\alpha_1$  at the top of the diagram of $\alpha$. We draw only one pair of such blocks (the yellow box and $B_1$) in Figure ~\ref{fig:i=1}, since the rows above  $B_1$ do not have any effect on the first row of $T_1$ and $T_2$ \svw{since the entries in these rows cannot be moved to the first row}. } \\

\svw{\textbf{Three observations:}  We now make three crucial observations on creating tableaux in $\QKT(\alpha)$.}

\begin{itemize}
\item[ i)]  The first row of a tableau in $\QKT(\alpha)$ is obtained by moving some of  the rightmost numbers in the lowest row of each block $B_m$ (of the basic tableau) down to the first row so that \svw{the contiguous cells that are moved down to the first row from} each block is apart by at least one empty cell. This follows by the definition of quasi-Yamanouchi Kohnert tableaux. For example, the gray cells in Figure ~\ref{fig:i=1} can be moved down to the first row.

\svw{As a consequence we have the next two observations.}

\item[ ii)] If an entry of the lowest row of $B_1$ is moved down to the first row of a tableau in $\QKT(\alpha)$, then  \svw{the only} possible  \svwii{moves} from rows $r >\mathrm{max}(B_2)$ in the basic \svwii{quasi-Yamanouchi} Kohnert tableau is to move entries either to the first row or within rows between $\mathrm{max}(B_2)+1$ and $\ell$. 

\item[ iii)] If an entry of the lowest row of $B_m$, $m>1$, is moved down to the first row of a tableau in $\QKT(\alpha)$, then   it is straightforward to verify that \svw{the only} possible \svwii{moves} from rows $r\geq \mathrm{min}(B_m)$ in the basic \svwii{quasi-Yamanouchi} Kohnert tableau is to move entries either down to the first row or no lower than the first row of $B_m$.    
\end{itemize}

\

\svw{\textbf{Concluding identical first rows:}} \svw{Towards a contradiction, suppose} that the first row of $T_1$ and $T_2$ are different and let $m$ be the smallest such that numbers of entries in the first row of $T_1$ and $T_2$ from $B_m$ are different. \svw{By our second and third observations we immediately get the following.} 
If $m=1$, then the number of entries in the last rows (above $B_2$) of $T_1$ and $T_2$ are  different and $T_1$ and $T_2$ cannot be of a same weight. 
If $m>1$, then the number of entries in $B_1\cup B_2\cup \cdots \cup B_{m}$ of $T_1$ and $T_2$ are different and $T_1$ and $T_2$ cannot be of a same weight.  \svwii{This contradicts that $T_1$ and $T_2$ are the same weight. Hence the first row of $T_1$ and $T_2$ are identical.}

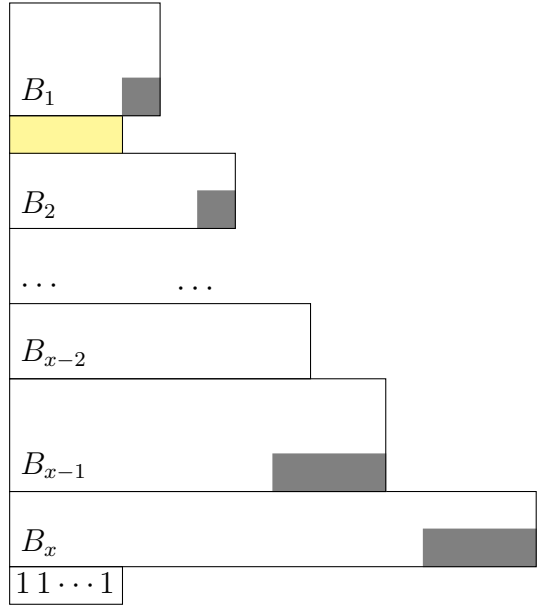
\begin{figure}[ht]
\begin{center}
    \begin{tikzpicture}% [xscale=1,yscale=2] 
    \draw (0,0) --(1.5,0)--(1.5,0.5)--(0,0.5)--cycle  node[above right] {$\!1\, 1 \cdots 1\,$};
    \filldraw[gray] (5.5,0.5) rectangle  (7,1);\draw (0,0.5)--(7,0.5)--(7,1.5)--(0,1.5)--cycle node[above right] {$B_x$}; 
    \filldraw[gray] (3.5,1.5) rectangle  (5,2);\draw (0,1.5)--(5,1.5)--(5,3)--(0,3)--cycle node[above right] {$B_{x-1}$}; 
    \draw (0,3)--(4,3)--(4,4)--(0,4)--cycle node[above right] {$B_{x-2}$}; 
    \draw (0,5)--(0,4) node[above right] {$\cdots$} ;
    \filldraw[gray] (2.5,5) rectangle  (3,5.5);\draw (0,5)--(3,5)--(3,6)--(0,6)--cycle  node[above right]{$B_2$};
    \filldraw[fill=yellow!50] (0,6)--(1.5,6)--(1.5,6.5)--(0,6.5)--cycle ;
   \filldraw[gray] (1.5,6.5) rectangle  (2,7); \draw (0,6.5)--(2,6.5)--(2,8)--(0,8)--cycle  node[above right]{$B_1$};
    \node at (2.5,4.2) {$\cdots$};
    \end{tikzpicture}
  \end{center}
  \caption{\label{fig:i=1} The basic tableau in $\QKT(\alpha)$ when $\alpha_1<\alpha_2$. }
\end{figure}

\item[$i>1$: ] First note that $\alpha_j\leq \alpha_{i+1}$ must hold for all $j>i+1$ since $\alpha$ satisfies Condition (a), and $\alpha_1\geq \alpha_2\geq \cdots \geq \alpha_i<\alpha_{i+1}$ by our assumption on $i$. 
There are two cases to consider:
 
If $\alpha_1\geq \alpha_{i+1}$, then the first row of any tableau in $\QKT(\alpha)$ must only contain $\alpha_1$ $1$s.

If not, then $\alpha_1<\alpha_{i+1}$ and only $(i+1)$s can be moved down to the first row of a quasi-Yamanouchi Kohnert tableau. This is because if $j>i+1$ can appear in the first row of a quasi-Yamanouchi Kohnert tableau, { then $\alpha_2\leq \alpha_1<\alpha_j<\alpha_{i+1}$, by the definition of quasi-Yamanouchi Kohnert \svwii{tableaux}, must hold} but this violates Condition (b). Therefore, two tableaux of the same weight must have the same first row consisting of $1$s and $(i+1)$s.
\end{itemize}

Now, we know that $T_1$ and $T_2$ have the same first row and we let $T_1'$ and $T_2'$ be the tableaux obtained by deleting the first row from $T_1$, $T_2$, respectively and subtracting $1$ from every remaining entry. Then $T_1'$ and $T_2'$ are quasi-Yamanouchi Kohnert tableaux of content $\beta$, where $\beta$ is obtained by subtracting the content $>1$ of the first row of $T_1$ (or $T_2$) from the relevant respective parts of $\alpha$. Then since there is still a cell in every row of the first column of $T_1'$ and $T_2'$ it follows that $\beta$ is a strong composition with $\ell(\alpha)-1$ parts. It is straightforward to check that $\beta$ does not contain any of the patterns (a), (b), (c). By the induction hypothesis, $T_1'$ and $T_2'$ must be the same tableau and hence we can conclude that $T_1$ and $T_2$ must be the same tableau too.
\end{proof}

When we restrict our attention to strong compositions $\alpha$, we are also able to give an algorithm to produce all the tableaux required to expand a key polynomial as a sum of fundamental slide polynomials, that is, produce all the elements of $\QKT(\alpha)$.
%%%%%%%%%%%%%%%%
\subsection{Recursive algorithm to produce all elements of $\QKT(\alpha)$}\label{subsec:alg} %for strong composition $\alpha$}

For a strong composition \svwii{$\alpha=(\alpha_1, \dots, \alpha_\ell)$,} let $T_\alpha$ be the basic \svwii{quasi-Yamanouchi} Kohnert tableau with content $\alpha$.
\begin{itemize} 
 \item  If $\mathrm{inv}(\alpha)=0$, then $\QKT(\alpha)=\{T_\alpha \}$.
   
 \item If $\mathrm{inv}(\alpha)=s>0$, then let $i$ be the smallest such that $\alpha_i<\alpha_{i+1}$ and let $\hat{\alpha}$  be the strong composition with $\mathrm{inv}(\hat{\alpha})=s-1$, obtained by interchanging the $i$th and the $(i+1)$th parts of $\alpha$; $\hat{\alpha}=(\alpha_1, \dots, \alpha_{i-1}, \alpha_{i+1}, \alpha_i, \dots, \alpha_\ell)$.
   
    For each $\hat{T}\in \QKT(\hat{\alpha})$, do
      \begin{itemize}
      \item  from columns $c=\alpha_i+1, \alpha_i+2, \dots $, change all $i$ into $i+1$ and change all $i+1$ into $i$, call the resulting tableau ${T}_0$; let $S(\hat{T}):=\{{T}_0 \}$
      
      \item  for $k=\alpha_i+1, \dots, \alpha_{i+1}$, do
        
         if the cell in row $(i+1) $ and column $k$ of  ${T}_0$ is empty {and  the cell in row $i$ and column $k$ contains $i+1$,} then swap the cells in row $i$ and row $i+1$ from column $\alpha_i+1$ to $k$ and note the resulting tableau ${T}_k$; let 
$S(\hat{T}):=S(\hat{T})\cup \{{T}_k \}$.
\end{itemize}
\end{itemize}

\begin{theorem}\label{thm:alg} For a strong composition $\alpha$, the above algorithm produces all quasi-Yamanouchi Kohnert tableaux of content $\alpha$, that is
 $$\QKT(\alpha)=\bigcup_{\hat{T}\in \QKT(\hat{\alpha})} S(\hat{T})\,.$$
\end{theorem}
\begin{proof}
It is easy to see that given any $T\in\QKT(\alpha)$, we can find a corresponding $\hat{T}\in\QKT(\hat{\alpha})$ by noting the rightmost $i+1$ in row $i+1$ in column $k$ and swapping all cells in row $i$ and row $i+1$ from column $\alpha_i+1$ to $k$. {Then for all columns $\alpha _i +1, \alpha _i +2, \ldots$ change all $i$ into $i+1$ and change all $i+1$ into $i$.}
\end{proof}

\begin{example}\label{ex:alg} Let $\alpha=(2,1,4,3)$. Then $i=2$ and $\hat{\alpha}=(2,4,1,3)$.

If we let 
$\hat{T}=\tableau
           {4&4
            \\ 3
            \\ 2&\red 2&4
             \\1&1&\red 2&\red 2 \\  \hline} \in \QKT(\hat{\alpha})$,\quad  then 
\,\,${T}_0=\tableau
           {4&4
            \\ 3& 
            \\ 2&\red 3&4
             \\1&1&\red 3&\red 3 \\  \hline} \in \QKT(\hat{\alpha})$.
    
 \medskip         
             
 {  The \svw{empty cell} in the third row and the second column of ${T}_0$, since the cell below it is filled with $i+1=3$,} will make the following  quasi-Yamanouchi Kohnert tableau of content $\alpha$.

$${T}_2=\tableau
           {4&4
            \\ 3& \blue 3 
            \\ 2&&4
             \\1&1&3&3 } \quad $$
%  {T}_3=\tableau
%           {4&4
 %           \\ 3&\blue 3 & \blue 4
 %           \\ 2&&
 %            \\1&1&3&3 \\  \hline} $$

\end{example}

\bigskip
%%%%%%%%%%%%%%%%%%%%%%%%%%%%%%%%%%%%%%%%%%%%%%%%%%%%%%%%%%%%%%
\section{Classifying when $\key_\a$ is multiplicity free for $\a$ a weak composition}\label{sec:MFforWeak}

At present classifying in general when a key polynomial is a multiplicity free expansion of fundamental slide polynomials seems substantially more complex than the strong composition case. However, we are able to make progress in some special cases. In particular we focus on the cases related to Theorem~\ref{thm:BvW13}, namely, hooks and two nonzero parts.

\begin{lemma}\label{lem:notMF} Let $\alpha$ be a strong composition with $\key_\alpha$ is not multiplicity free and $\a$ be a weak composition with $\mathrm{flat}(\a)=\alpha$. Then $\key_\a$ is not multiplicity free.
\end{lemma}
\begin{proof} If there are two different tableaux of the same weight in $\QKT(\alpha)$, then by adding empty rows and increasing the numbers in the cells by the number of empty rows inserted below we will produce two different tableaux of the same weight in $\QKT(\a)$.
\end{proof}

\begin{lemma}\label{lem:MF} If $s_{\mathrm{sort}(\a)}$ is multiplicity free, then $\key_\a$ is multiplicity free.
\end{lemma}
\begin{proof} Suppose that $\key_\a$ is not multiplicity free. Then, similar to the proof of Lemma~\ref{lem:notMF}, certainly $\key_{0^m\times \a}$ is not multiplicity free for all $m$ and we can conclude that  $s_{\mathrm{sort}(\a)}$ is not multiplicity free due to Theorem~\ref{thm:limit}.
\end{proof}

\begin{corollary} \label{cor:hooks} Let $\a$ be a weak composition with $\mathrm{sort}(\a) = (n-k, 1^k)$. Then $\key_\a$ is multiplicity free.
\end{corollary}

\subsection{When $\a$ has two nonzero parts}\label{subsec:two nonzero parts}

To give an idea of how the complexity increases with the relaxing from strong compositions to weak compositions, note that this case splits into three, depending on whether the number of leading $0$s is at least two, one or none.

\begin{theorem}\label{thm:two_parts} For a weak composition $\a$ with two nonzero parts  and at least two leading $0$s, $\key_\a$ is multiplicity free if and only if
\begin{enumerate}
\item[(1)] $\mathrm{sort}(\a)=\lambda$ for the partitions $\lambda= (3, 3), (4, 4), (n-2, 2)$ for $n\geq 4$, $(n-1, 1)$ for $n \geq 2$, or
\item[(2)] $\mathrm{flat}(\a)=(4, 3)$. 
\end{enumerate}
\end{theorem}
\begin{proof} Lemma~\ref{lem:MF} together with Theorem~\ref{thm:BvW13} shows that $\key_\a$ is multiplicity free if  $\mathrm{sort}(\a)$ is one of  $(3, 3), (4, 4), (n-2, 2)$ or $(n-1, 1)$. We can check by hand that $\key_\a$ for $\a=(0,0,4,3)$ \svw{and $\a=(0,0,0,4,3)$} is multiplicity free. Moreover, more zeros at the front will keep the key polynomial multiplicity free  since by the definition of quasi-Yamanouchi Kohnert tableaux no cell can move to these \svw{added} rows with $0$ cells. Similarly adding zeros between the nonzero parts will keep the key polynomials multiplicity free. Hence, if $\mathrm{flat}(\a)=(4, 3)$ then $\key_\a$ is multiplicity free.

{To prove the other direction, for each of the following Cases (1)-(4) we explicitly find two different tableaux of same weight in $\QKT(0,0,\alpha_1, \alpha_2)$.  We note that adding zeros at the front does not reduce the multiplicity. Moreover, all tableaux we give can be naturally extended when we add zeros between $\alpha_1$ and $\alpha_2$, which shows the general cases also. }

\begin{itemize}
    \item[(1)] $\alpha_1=\alpha_2 \geq 5$; see Figure~\ref{fig:case(1)}.
    \item[(2)] $\alpha_1>\alpha_2 \geq 4$; see Figure~\ref{fig:case(2)}.
    \item[(3)] $\alpha_1-1>\alpha_2=3$; see Figure~\ref{fig:case(3)}. 
 { \item[(4)] $\alpha_2>\alpha_1 \geq 3$; see Figure~\ref{fig:case(4)}.}
  %  \item[(5)] $\alpha_1=3,\, \alpha_2= 4$; see Figure~\ref{fig:case(5)}.
\end{itemize}
   
    \begin{figure}[ht]
\begin{center}
    \begin{tikzpicture}[xscale=2,yscale=2] 
       \node at (0,0)  {$ \vline\tableau{\\4&4\\3& & 4 & 4 \\ &3 &3 & & 4&4&\cdots &4\\ & & & 3& 3&3& \cdots &3\\ \hline} $}; 
       \node at (3.5,0)  {$ \vline\tableau{\\4&4\\3& 3 & 4 \\ & &3 &4& 4&4&\cdots &4\\ & & & 3& 3&3&\cdots &3\\ \hline } $};
\end{tikzpicture}
  \end{center}
  \caption{\label{fig:case(1)}  Two tableaux in $\QKT(0,0,\alpha_1,\alpha_2)$  of same weight when $\alpha_1=\alpha_2\geq 5$.}
\end{figure}
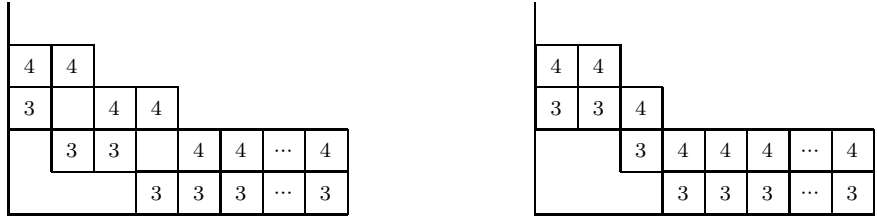

     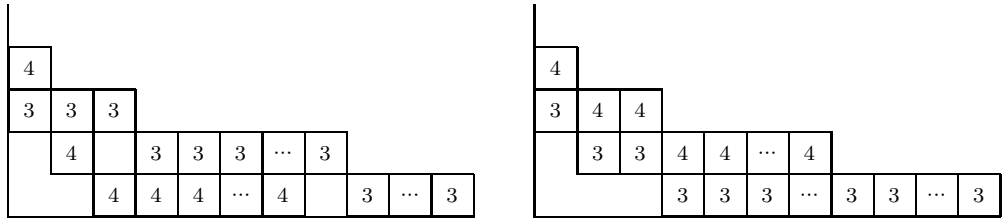
\begin{figure}[ht]
\begin{center}
    \begin{tikzpicture}[xscale=2,yscale=2] 
       \node at (0,0)  {$ \vline\tableau{\\4\\3&3 &3 \\ &4& &3 &3&3&\cdots&3\\ & &4 &4&4&\cdots&4&&3&\cdots&3\\ \hline} $};
       \node at (3.5,0)  {$ \vline\tableau{\\4\\3&4 &4 \\ &3& 3 &4&4&\cdots&4\\ && &3 &3&3&\cdots &3&3&\cdots&3\\ \hline} $};
\end{tikzpicture}
  \end{center}
  \caption{\label{fig:case(2)}  Two tableaux in $\QKT(0,0,\alpha_1,\alpha_2)$  of  same weight when $\alpha_1>\alpha_2\geq 4$.}
\end{figure}

      \begin{figure}[ht]
\begin{center}
    \begin{tikzpicture}[xscale=2,yscale=2] 
       \node at (0,0)  {$ \vline\tableau{\\3&3 &3 \\4& &&3 &3\\ & 4 &4& & & 3 & \cdots & 3\\ \hline} $};
       \node at (3.5,0)  {$ \vline\tableau{\\3&3 &3 \\4&4&&3 \\& & 4& &3& 3 & \cdots & 3\\ \hline} $};
\end{tikzpicture}
  \end{center}
  \caption{\label{fig:case(3)}  Two tableaux in $\QKT(0,0,\alpha_1,\alpha_2)$  of  same weight when $\alpha_1-1>\alpha_2=3$.}
\end{figure}
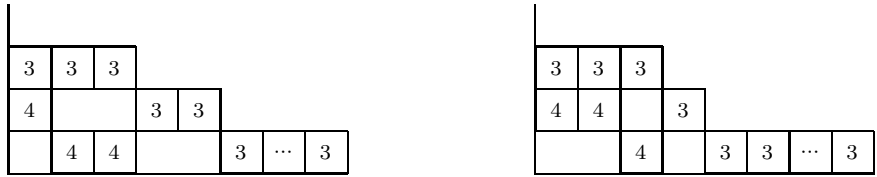
  
    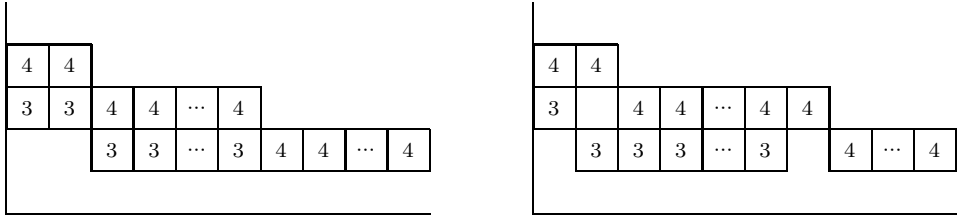
\begin{figure}[ht]
\begin{center}
    \begin{tikzpicture}[xscale=2,yscale=2] 
       \node at (0,0)  {$ \vline\tableau{\\4&4 \\3&3&4&4&\cdots&4\\& &3 &3&\cdots&3 &4&4&\cdots & 4\\  \\ \hline}$};
       \node at (3.5,0)  {$ \vline\tableau{\\4&4 \\3&&4&4&\cdots & 4&4\\& 3 &3&3&\cdots & 3& & 4&\cdots & 4\\ \\ \hline} $};
\end{tikzpicture}
  \end{center}
  \caption{\label{fig:case(4)} Two tableaux in $\QKT(0,0,\alpha_1,\alpha_2)$  of  same weight when $\alpha_2>\alpha_1\geq 3$.}
\end{figure}
  \end{proof}  
      
 % \begin{figure}[ht]
%\begin{center}
 %   \begin{tikzpicture}[xscale=2,yscale=2] 
  %     \node at (0,0)  {$ \vline\tableau{\\4&4 \\3&3&4\\ & &3 &4\\ \\ \hline} $};
   %    \node at (2,0)  {$ \vline\tableau{\\4&4 \\3&&4&4\\ & 3 &3\\ \\ \hline} $};
%\end{tikzpicture}
 % \end{center}
 % \caption{\label{fig:case(5)}  Two tableaux in $\QKT(0,0,3, 4)$ of a same weight.}
%\end{figure}

{In the proof of Theorem~\ref{thm:two_parts}, two given quasi-Yamanouchi Kohnert tableaux of the same weight for  Case (4) given in Figure~\ref{fig:case(4)}, have \emph{empty} first rows. This means that there are at least two tableaux in $\QKT(0, \alpha_1, \alpha_2)$ when $\alpha_1, \alpha_2$ satisfy Case (4). % we can have some useful facts on weak compositions that has at least one zero at front.

\begin{lemma}\label{lem:two_parts_1} Let $\a$ be a weak composition with at least one leading $0$ and $\mathrm{flat}(\a)=(\alpha_1, \alpha_2)$ for $\alpha_2>\alpha_1\geq 3$. Then $\key_\a$ is not multiplicity free.
\end{lemma}

\begin{theorem}\label{thm:two_parts_2}  For a weak composition $\a$ with two nonzero parts  and exactly one leading $0$, $\key_\a$ is multiplicity free if and only if
\begin{enumerate}
\item[(1)] $\mathrm{sort}(\a)=\lambda$ for the partitions $\lambda= (3, 3), (4, 4), (n-2, 2)$ for $n\geq 4$, $(n-1, 1)$ for $n\geq 2$, or
\item[(2)] $\mathrm{flat}(\a)=(\alpha_1, \alpha_2)$, for $\alpha_1\geq \alpha_2$.
\end{enumerate}
\end{theorem}
\begin{proof}
We know from Lemma~\ref{lem:two_parts_1}, that  cases other than the ones in the theorem have multiplicities. Hence, we only need to show that in Cases (1) or (2), $\key_\a$ is multiplicity free.

Lemma~\ref{lem:MF} together with Theorem~\ref{thm:BvW13} shows that $\key_\a$ is multiplicity free if  $\mathrm{sort}(\a)$ is one of  $(3, 3), (4, 4), (n-2, 2)$ or $(n-1, 1)$. 

Before we prove that weak compositions in (b) are multiplicity free, we note the following: If $T$ is a tableau in $\QKT(0,\alpha_1, \underbrace{0,\dots, 0}_m, \alpha_2)$ for $m>0$, then the row $i$,  $i=3, \dots, 2+m$, of $T$ is empty, due to the definition of quasi-Yamanouchi Kohnert tableaux. Hence it is enough to consider the weak compositions $(0, \alpha_1, \alpha_2)$ for the proof. 
 
Suppose that $\alpha_1=\alpha_2=a$. By the definition of (quasi-Yamanouchi) Kohnert tableaux Part (iv) we know that none of the $a$ 3s can appear in row 1. Thus row 1 will contain $x$ 2s, row 3 will contain $y$ 3s, and row 2 will contain the remaining 2s and 3s. Since the weight of each of these tableaux is uniquely determined by the number of cells in row 1 and in row 3, it follows that $\kappa _{(0,a,a)}$ is multiplicity free.

Now suppose that $\alpha_1>\alpha_2\geq 4$, or $\alpha_1\geq 5$ and $\alpha_2=3$. If only $x$, for $1\leq x \leq \alpha_1$, cells from row 2 move down to row 1, then we do not obtain a quasi-Yamanouchi Kohnert tableau.
If only $y$, for $0\leq y \leq \alpha_2$, cells from row 3 move down to row 1, then we obtain a quasi-Yamanouchi Kohnert tableau of weight
$$(y, \alpha_1, \alpha_2-y).$$ If cells from rows 2 and 3 move down to row 1, then, by definition, row 2 has at least $\alpha_2+1$ cells but $< \alpha _ 1$ cells, and row 3 has $<\alpha_2$ cells. Lastly, if cells from row 2 move down to row 1, and cells from row 3 move down to row 2, then, by definition, row 2 has at most $\alpha_2$ cells and row 3 has $<\alpha_2$ cells. In every case the weight is unique, and hence  $\kappa _{(0,\alpha_1,\alpha_2)}$ is multiplicity free.

Finally, $\kappa _{(0, 4, 3)}$ is multiplicity free since $\kappa _{(0, 0, 4, 3)}$ is multiplicity free as we showed in Theorem~\ref{thm:two_parts}.
\end{proof}
}

\begin{theorem}\label{thm:two_parts_3} For a weak composition $\a$ with two nonzero parts  and no leading $0$, $\key_\a$ is multiplicity free.\end{theorem}

\begin{proof} Let $\mathrm{flat}(\a)=(\alpha_1, \alpha_2)=\alpha$. If $\alpha _1 \geq \alpha _2$, then $\key _\a$ is multiplicity free by the third part of Theorem~\ref{thm:k=f}. If $\alpha _1 < \alpha _2$, then $\key _\alpha$ is multiplicity free by Lemma~\ref{lem:expansion}. Adding zeros between the nonzero parts will keep the key polynomial multiplicity free by definition, and so $\key _\a$ is multiplicity free. 
\end{proof}

%%%%%%%%%%%%%%%%%%%%%%%%%%%%%%%%%%%%%%%%%%%%%%%%%%%%%%%%%%%%%%%%%%%%%%%%%%%%
\bibliographystyle{amsplain}
\bibliography{key_paper}
\end{document}